\documentclass[12pt,a4paper]{amsart}
\usepackage{a4wide,amsthm,amsfonts,amssymb,amscd,mathrsfs,amsaddr,mathtools,amstext}
\usepackage[all]{xy}

\newcommand\Hom{\operatorname{Hom}\nolimits}

\newcommand\im{\operatorname{Im}\nolimits}

\newcommand\rk{\text{rk}}

\renewcommand\emptyset{\varnothing}

\newcommand\lrdash{\:
\xymatrix@1{\ar@{-->}[r]&}\:
}
\newcommand{\lrdashar}[1]{\:
\xymatrix@1{\ar@{-->}[r]^{#1}&}\:
}
\newcommand{\sendsto}[1]{\:
\xymatrix@1{\ar@{|->}[r]^{#1}&}\:
}
\newcommand{\intoo}[1]{\:
\xymatrix@1{\ar@{^(->}[r]^{#1}&}\:
}

\newtheorem{theorem}{Theorem}[section]
\newtheorem{proposition}[theorem]{Proposition}

\newtheorem{lemma}[theorem]{Lemma}

\newtheorem{definition}[theorem]{Definition}

\theoremstyle{definition}

\newtheorem{proposition-definition}[theorem]{Proposition-Definition}
\newtheorem{rem}[theorem]{Remark}

\newenvironment{remark}{\begin{rem}}{\hfill $\triangle$\end{rem}}

\begin{document}

\title[Symplectic instantons   on $\mathbb{P}^3$]{Symplectic instanton   bundles\\[5pt] on  $\mathbb{P}^3$ and 't Hooft
instantons}

\author{\small U. Bruzzo,$^{\P\star}$ D. Markushevich$^\S$ and A.~S. Tikhomirov$^{\dag}$}

\address{\small\rm $^\P$ Scuola Internazionale Superiore di Studi Avanzati (SISSA), \\
Via Bonomea 265, 34136 Trieste, Italia \\[3pt]
$^\star$Istituto Nazionale di Fisica Nucleare, Sezione di Trieste \\[3pt]
$^\S$ Math\'ematiques -- b\^{a}t.~M2, Universit\'e Lille 1, \\
F-59655 Villeneuve d'Ascq Cedex, France \\[3pt]
$^\dag$Department of Mathematics, Higher School of Economics, \\ 
7 Vavilova Str., 117312 Moscow, Russia}

\email{bruzzo@sissa.it, markushe@math.univ-lille1.fr, astikhomirov@mail.ru}

\begin{abstract}
We study the moduli space
$I_{n,r}$ of rank-$2r$ symplectic instanton vector bundles on $\mathbb{P}^3$ with $r\ge2$ and second Chern class
$n\ge r+1,\ n-r\equiv 1(\mathrm{mod}2)$.
We introduce the notion of tame symplectic instantons by excluding a kind of pathological monads and show that the
locus $I^*_{n,r}$ of tame symplectic instantons is irreducible and has the expected dimension, equal to $4n(r+1)-r(2r+1)$. The proof is inherently based on a relation between the spaces $I^*_{n,r}$ and the moduli spaces of 't Hooft instantons.
\end{abstract}

\thanks{D.\ M.\ and A. T. were partially supported by Labex CEMPI (ANR-11-LABX-0007-01),
and U.\ B.\ by PRIN ``Geometria delle variet\`a algebriche" and INdAM-GNSAGA. U.\ B.\ is a member of the {\sc vbac} group.
A. T. acknowledges the hospitality of the Max-Planck-Institut f\"ur Mathematik in Bonn and SISSA in Trieste, where   part of the work on this paper was made.}

\maketitle

\thispagestyle{empty}

\section{Introduction}\label{sec0}
A {\it symplectic instanton vector bundle} of rank $2r$ and charge $n$
on the projective 3-space $\mathbb{P}^3$ is
an algebraic vector bundle $E=E_{2r}$  of rank $2r$ on $\mathbb{P}^3$
which is equipped with a symplectic structure $\phi:\ E\overset{\sim}\to E^\vee, \phi^\vee=-\phi$
and satisfies the vanishing conditions
$h^0(E)=h^1(E\otimes\mathcal{O}_{\mathbb{P}^3}(-2))=0$.
The Chern  classes
$c_1(E)$ and $c_3(E)$ vanish, and we also assume $c_2(E)=n\ge1$.
We shall denote by $I_{n,r}$ the moduli space of symplectic $(n,r)$-instantons.

Rank $r$ symplectic instantons on $\mathbb P^3$ relate in a natural manner with ``physical'' $\mathbf{Sp}(r)$
instantons on the four-sphere $S^4$, i.e., connections on principal $\mathbf{Sp}(r)$-bundles on $S^4$ with self-dual
curvature \cite{A}; the moduli spaces of the former are in a sense a complexification of the moduli spaces of the
latter. This relation is expressed by the so-called Atiyah--Ward correspondence \cite{AW,A}, which relies on the fact
that the projective space $\mathbb P^3$ is the twistor space of the four-sphere $S^4$. The present paper and its
companion \cite{BMT1} are the first to study the geometry of the moduli spaces $I_{n,r}$. While \cite{BMT1} studied
the case $n \equiv r$ (mod 2), with $n \ge r$, the present paper deals with the other case, $n\equiv r+1$ (mod 2),
with $n \ge r+1$.  The main result of this paper is that a component $I^\ast_{n,r}$ of
$I_{n,r}$ that is singled out by a certain open condition (which rules out some ``badly behaved'' monads) is
irreducible.

We exploit as usual the monad method \cite{H,ADHM,B,B1,BH,Tju1,Tju2}, which allows one to study instantons by means
of  hyperwebs of quadrics. Namely, we realize $I_{n,r}$ as the quotient space of a principal
$GL(H_n)/\{\pm{\rm id}\}$-bundle $\pi_{n,r}: MI_{n,r}\to I_{n,r}$, where $MI_{n,r}$ is a locally closed subset of
the vector space $\mathbf{S}_n$ of hyperwebs of quadrics (precise definitions will be given later on). The tame locus $I^\ast_{n,r}$ being open in $I_{n,r}$,
its irreducibility is equivalent to that of
$MI^\ast_{n,r}=\pi_{n,r}^{-1}(I^\ast_{n,r})$.
The key ingredient of our approach is the reduction of the last problem to that of certain sets $Z_{n-r+1}$
(see section \ref{I^*nr}). The sets $Z_i$ as locally closed subsets of some vector spaces related to $\mathbf{S}_n$
were first defined in \cite{T1}. It is shown in \cite[Section 9]{T1} that the $Z_i$ can be interpreted
essentially as open subsets of certain affine bundles over the monad spaces $M^{tH}_{2i-1}$ of 't
Hooft rank-2 mathematical instantons of charge $2i-1$---see more details in section 3.2. Thus the irreducibility of $Z_{n-r+1}$, hence that of $I^\ast_{n,r}$, is reduced to the irreducibility of the moduli spaces of 't Hooft instantons
of fixed charge, which is well known; see references in \cite{T1}. This nontrivial relation between the
spaces $I^\ast_{n,r}$ and the moduli of 't Hooft instantons is crucial for the results in this paper. Note
that this process of reduction from $I^\ast_{n,r}$ to the moduli of 't Hooft instantons somewhat resembles
Barth's approach in \cite{B1} to the proof of the irreducibility of the moduli space $I_4$ of
instantons of charge 4. In that paper, Barth reduces the problem to the irreducibility of the space $Q_n$ of
commuting pairs of (good in some sense) pencils of quadrics for $n=4$. In our case the role of the spaces $Q_n$
is played by the moduli spaces of 't Hooft instantons.

\subsection*{Acknowledgements}   This paper was partly written while the first author was visiting
Universit\'e Lille I. He thanks  the Department of Mathematics of  Universit\'e Lille I for
hospitality and support. The second and the third authors acknowledge the hospitality of the Max-Planck-Institut für Mathematik in Bonn, were they made a part of work on the paper. The third author thanks the Ministry of Education and Science of the Russian Federation for
partial support.

\subsection*{Notation and conventions}
Throughout this paper, we consider an  algebraically closed
  base field $\Bbbk$   of characteristic 0. All   schemes  will be Noetherian. By a general point of an irreducible
(but not necessarily reduced) scheme $\mathcal{X}$ we mean a closed point of a dense open subset of
$\mathcal{X}$. An irreducible scheme is  generically reduced if it is reduced at all general points.
  We follow the notation
of \cite{T1}. So, we fix an integer $n\ge1$, and denote by $H_n$ and $V$  fixed vector spaces over $\Bbbk$ of dimension $n$ and 4, respectively,   and set $\mathbb{P}^3=P(V)$. Furthermore,  $\mathbf{S}_n$
(the \textit{space of hyperwebs of quadrics})
will denote the vector space $S^2H_n^\vee\otimes \wedge^2V^\vee$.   A hyperweb of quadrics $A\in\mathbf{S}_n$ is a
skew-symmetric homomorphism $A:H_n\otimes V\to H_n^\vee\otimes V^\vee$, and we denote by $W_A$ the vector space
$H_n\otimes V/{\ker A}$ and by $c_A$ the canonical epimorphism $H_n\otimes V\twoheadrightarrow W_A$. A choice of
$A$ induces a skew symmetric isomorphism $q_A:W_A\overset{\sim}\to W_A^{\vee}$,
and $A$ is the  composition
$H_n\otimes V\overset{c_A}\twoheadrightarrow W_A\underset{\sim}{\overset{q_A}\to} W_A^{\vee}\overset{c_A^\vee}\rightarrowtail
H_n^\vee\otimes V^\vee$.

For any morphism of
$\mathcal{O}_X$-sheaves $f:\mathcal{F}\to\mathcal{F}'$
we denote by the same letter $f$ the induced morphism
$id\otimes f:U\otimes\mathcal{F}\to U\otimes\mathcal{F}'$, and analogously,
for any homomorphism $f:U\to U'$ of $\Bbbk $-vector spaces, the induced morphism $f\otimes id:U\otimes\mathcal{F}\to U'\otimes\mathcal{F}$.
For $A\in\mathbf{S}_n$ we denote by $a_A$ the composition
$H_n^\vee\otimes\mathcal{O}_{\mathbb{P}^3}(-1)\overset{u}\to H_n\otimes V\otimes\mathcal{O}_{\mathbb{P}^3}
\overset{c_A}\to{W}_{A}\otimes\mathcal{O}_{\mathbb{P}^3}$,
where $u$ is the tautological subbundle morphism.
By abuse of notation, we denote by the same symbol a $\Bbbk $-vector space, say $U$, and the associated affine space
$\mathbf{V}(U^\vee)={\rm Spec}({\rm Sym}^*U^\vee)$.

\bigskip
\section{Explicit construction of symplectic instantons}\label{explicit}

In this section we provide some examples and  recall some facts about $MI_{n,r}$, in particular, its  relation with the moduli space $I_{n,r}$ of symplectic instantons, see \cite[Section 3]{BMT1}.
Let us consider the {\it set of $(n,r)$-instanton hyperwebs of quadrics}
\renewcommand\theenumi{\roman{enumi}}
\begin{equation}\label{space of nets}
MI_{n,r}:=\left\{A\in \mathbf{S}_n\ \left|\
\begin{minipage}{26em}
\begin{enumerate}
\item $\rk(A:H_n\otimes V\to H_n^\vee\otimes V^\vee)=2n+2r$,
\item the morphism $a_A^\vee:W^\vee_{A}\otimes\mathcal{O}_{\mathbb{P}^3}\to
H_n^\vee\otimes\mathcal{O}_{\mathbb{P}^3}(1)$
is surjective,
\item $h^0(E_{2r}(A))=0$, where $E_{2r}(A):=\ker(a^\vee_A\circ q_{A})/\im a_A$.
\end{enumerate}
\end{minipage}\right.
\right\}
\end{equation}

\begin{theorem}\label{princbundle}
(i) For each $n\ge1$, the space $MI_{n,r}$ of $(n,r)$-instanton nets of quadrics is a locally closed subscheme
of the vector space $\mathbf{S}_n$, given locally at any point
$A\in MI_{n,r}$ by
\begin{equation}\label{eqns(i)}
\binom{2n-2r}{2}=2n^2-n(4r+1)+r(2r+1)
\end{equation}
equations obtained as the rank condition (i) in \eqref{space of nets}.

(ii) The natural morphism
\begin{equation}\label{pi nr}
\pi_{n,r}: MI_{n,r}\to I_{n,r},\ A\mapsto[E_{2r}(A)],
\end{equation}
is a principal
$GL(H_n)/\{\pm{\rm id}\}$-bundle in the \'etale topology.
Hence $I_{n,r}$ is a  quotient stack $MI_{n,r}/(GL(H_n)/\{\pm{\rm id}\})$, and is therefore an algebraic space.
\end{theorem}

The fibre $F_{[E]}=\pi_n^{-1}([E])$ over a  point $[E]\in I_{n,r}$ is a principal homogeneous
space of    $GL(H_n)/\{\pm{\rm id}\}$, so that  the
irreducibility of $(I_{n,r})_{red}$  amounts to the irreducibility of the scheme $(MI_{n,r})_{red}$.
Besides, \eqref{eqns(i)} yields
\begin{equation}\label{dim MInr ge...}
\dim_AMI_{n,r}\ge\dim \mathbf{S}_n-(2n^2-n(4r+1)+r(2r+1))=n^2+4n(r+1)-r(2r+1)
\end{equation}
at all points $A\in MI_{n,r}$. Thus, $\dim_{[E]}I_{n,r}\ge4n(r+1)-r(2r+1)$
at all points $[E]\in I_{n,r}$, as $MI_{n,r}\to I_{n,r}$ is an \'etale  principal
$GL(H_n)/\{\pm{\rm id}\}$-bundle.

\subsection{Symplectic $(n+1,n)$-instantons}
\label{n+1,n-instantons}
We give a construction of symplectic $(n+1,n)$-instantons and describe  their relation to usual rank-2
instantons with   second Chern class $c_2=2n$. This will be established at the level of spaces of hyperwebs of
quadrics $MI_{n+1,n}$ and $MI_{2n,1}$,  regarded as spaces of monads.

Denote by $\mathrm{Isom}_{n+1,n-1}$ the set of all isomorphisms
\begin{equation}\label{zeta}
\zeta:H_{n+1}\oplus H_{n-1}\overset{\sim}\to H_{2n}.
\end{equation}
This  is the
principal homogeneous space of the group $GL(2n)$.
Moreover, for any $\zeta\in\mathrm{Isom}_{n+1,n-1}$, let
$p_{\zeta}:\mathbf{S}_{2n}\twoheadrightarrow\mathbf{S}_{n+1}$
be the induced epimorphism, and, for any monomorphism $i:H_n\hookrightarrow H_{n+1}$, let
$pr_{(i)}:\mathbf{S}_{n+1}\to\mathbf{S}_n$ be the induced
epimorphism.

Note that $MI_{2n,1}$ is irreducible \cite[Theorem 1.1]{T2}, and
one has the following result \cite[Theorem 3.1]{T2}.
\begin{theorem}\label{rk4n+2}
There exists a dense open subset $MI_{2n,1}^*$ of $MI_{2n,1}$
such that, for any hyperweb $A\in MI_{2n,1}^*$ and a general
$\zeta\in\mathrm{Isom}_{n+1,n-1}$ the rank of the homomorphism
$B=p_{\zeta}(A):H_{n+1}\otimes V\to H_{n+1}^\vee\otimes V^\vee$
coincides with the rank of
$A:H_{2n}\otimes V\to H_{2n}^\vee\otimes V^\vee$:
\begin{equation}\label{rkrk}
\mathrm{rk}B=\mathrm{rk}A=4n+2.
\end{equation}
\end{theorem}
Set $W_{4n+2}:=H_{2n}\otimes V/\ker A$ and define the skew-symmetric isomorphism
$q_A:\ W_{4n+2}\overset{\sim}\to W_{4n+2}^\vee$
and the morphism of sheaves
$a_A:\ H_{2n}\otimes\mathcal{O}_{\mathbb{P}^3}(-1)\to W_{4n+2}\otimes\mathcal{O}_{\mathbb{P}^3}$
with $H_{2n}$ and $W_{4n+2}$ taken instead of $H_n$ and $W_A$, respectively.
The morphism $a_A$ and its transpose
${}^ta_A=a_A^\vee\circ q_A:\ W_{4n+2}\otimes\mathcal{O}_{\mathbb{P}^3}\to
H_{2n}^\vee\otimes\mathcal{O}_{\mathbb{P}^3}(1)$
yield a monad
$$
\mathcal{M}_{A}:\ \ 0\to H_{2n}\otimes\mathcal{O}_{\mathbb{P}^3}(-1)\overset{a_A}\to
W_{4n+2}\otimes\mathcal{O}_{\mathbb{P}^3}\overset{{}^ta_A}\to
H_{2n}^\vee\otimes\mathcal{O}_{\mathbb{P}^3}(1)\to0
$$
with   cohomology sheaf $E(A),\ [E(A)]\in I_{2n,1}$, see Theorem \ref{princbundle}.

Let
$$
i_\zeta:H_{n+1}\hookrightarrow H_{2n}
$$
be the monomorphism defined by the isomorphism \eqref{zeta}.
The composition
$a_B:\ H_{n+1}\otimes\mathcal{O}_{\mathbb{P}^3}(-1)
\overset{i_\zeta}\hookrightarrow H_{2n}\otimes
\mathcal{O}_{\mathbb{P}^3}(-1)\overset{a_A}\to W_{4n+2}\otimes
\mathcal{O}_{\mathbb{P}^3}$
and its transpose
${}^ta_B=a_B^\vee\circ q_A$
yield a monad
$$\mathcal{M}_{B}:\ \ 0\to H_{n+1}\otimes\mathcal{O}_{\mathbb{P}^3}(-1)\overset{a_B}
\to W_{4n+2}\otimes\mathcal{O}_{\mathbb{P}^3}\overset{{}^ta_B}
\to H_{n+1}^\vee\otimes\mathcal{O}_{\mathbb{P}^3}(1)\to0
$$
with the cohomology sheaf
$$
E_{2n}(B):=\ker{}^ta_B/{\rm im}a_B,\ \ \ c_2(E_{2n}(B))=n+1.
$$
The symplectic isomorphism
$q_A:\ W_{4n+2}\overset{\sim}\to W_{4n+2}^\vee$ induces a
symplectic structure on $E_{2n}(B)$,
\begin{equation}
\phi_B: E_{2n}(B)\overset{\sim}{\to}E_{2n}(B)^\vee.
\label{phiB}
\end{equation}
Moreover, \eqref{rkrk} implies an isomorphism
$H_{n+1}\otimes V/\ker B\simeq W_{4n+2}$,
hence a monomorphism of spaces of sections
$h^0({}^ta_B):W_{4n+2}\otimes\mathcal{O}_{\mathbb{P}^3}
\overset{{}^ta_B}\to H_{n+1}^\vee V^\vee$
in the monad $\mathcal{M}_{B}$. Hence for this monad
one has  $h^0(E_{2n}(B))=0$. This together with \eqref{phiB} means
that $E_{2n}(B)$ is a symplectic instanton:
$$
[E_{2n}(B)]\in I_{n+1,n}.
$$
Note that by construction the monads
$\mathcal{M}_A$ and $\mathcal{M}_{B}$
fit into the commutative diagram
\begin{equation}\label{twomonads}
\xymatrix{
0\ar[r] & H_{n+1}\otimes\mathcal{O}_{\mathbb{P}^3}(-1)
\ar[r]^-{a_B}\ar@{^{(}->}[d]^{i_\zeta} &
W_{4n+2}\otimes\mathcal{O}_{\mathbb{P}^3}\ar[r]^{q_A}_{\cong}
\ar@{=}[d]_{\cong} &
W_{4n+2}^\vee\otimes\mathcal{O}_{\mathbb{P}^3}\ar[r]^-{a_B^\vee}
& H_n^\vee\otimes\mathcal{O}_{\mathbb{P}^3}(1)\ar[r] & 0 \\
0\ar[r] & H_{2n}\otimes\mathcal{O}_{\mathbb{P}^3}(-1)
\ar[r]^-{a_A} & W_{4n+2}\otimes\mathcal{O}_{\mathbb{P}^3}
\ar[r]^{q_A}_{\cong} & W_{4n+2}^\vee\otimes
\mathcal{O}_{\mathbb{P}^3}\ar[r]^-{a_A^\vee}
\ar@{=}[u]^{w^\vee}_{\cong} &
H_{2n}^\vee\otimes\mathcal{O}_{\mathbb{P}^3}(1)\ar[r]\ar@{->>}
[u]^{i_\zeta^\vee} & 0,}
\end{equation}
In view of
\eqref{phiB} and the canonical isomorphism $H_{2n}/i_\zeta(H_{n+1})\simeq H_{n-1}$,  this
diagram yields the quotient monad
$$
\mathcal{M}_{A,B}:\ \ 0\to H_{n-1}\otimes
\mathcal{O}_{\mathbb{P}^3}(-1)\overset{a_{A,B}}\to E_{2n}(B)
\overset{\phi_B}{\underset{\simeq}\to}E_{2n}(B)^\vee
\overset{a_{A,B}^\vee}\to H_{n-1}^\vee\otimes
\mathcal{O}_{\mathbb{P}^3}(1)\to0
$$
whose cohomology sheaf is
$$
E_2(A)=\ker(a_{A,B}^\vee\circ \phi_B)/{\rm im\ } a_A.
$$

\subsection{A special family of symplectic
$(2n-r+1,r)$-instantons}\label{nr-instantons}
For any integer $r,\ 2\le r\le n-1$,
with $n\ge3$,
consider a monomorphism
\begin{equation}\label{monotau}
\tau:H_{2n-r+1}\hookrightarrow H_{2n}
\end{equation}
such that
\begin{equation}\label{contains}
\tau(H_{2n-r+1})\supset i_\zeta(H_{n+1}).
\end{equation}
The image of $A\in MI_{2n,1}$ under the projection
$\mathbf{S}_{2n}\twoheadrightarrow\mathbf{S}_{2n-r+1}$ induced
by $\tau$ produces
 a hyperweb of quadrics
$$
A_\tau\in\mathbf{S}_{2n-r+1}.
$$
 This corresponds to a monad
$$
\mathcal{M}_\tau:\ \ 0\to H_{2n-r+1}\otimes
\mathcal{O}_{\mathbb{P}^3}(-1)\overset{a_\tau}\to W_{4n+2}
\otimes\mathcal{O}_{\mathbb{P}^3}\overset{a_\tau^\vee\circ q_A}
\to H_{2n-r+1}^\vee\otimes\mathcal{O}_{\mathbb{P}^3}(1)\to0,
$$
whose cohomology is the  rank $2r$ bundle
\begin{equation}\label{rk2n coho sheaf}
E_{2r}(A_\tau)=\ker(a_\tau^\vee\circ q_A)/{\rm im\:} a_\tau.
\end{equation}
where $a_\tau:=a_A\circ\tau$.  The bundle $E_{2r}(A_\tau)$ has a natural symplectic structure
\begin{equation}\label{sympl E2r}
\phi_r:\ E_{2r}(A_\tau)\overset{\sim}\to E_{2r}(A_\tau)^\vee
\end{equation}
induced by the antiselfduality of the monad $\mathcal{M}_{\tau}$.
Moreover by \eqref{contains} the monad
$\mathcal{M}_\tau$ can be included into
diagram \eqref{twomonads}
 as a middle row, thus obtaining   a   three-row
commutative, anti-self-dual diagram. Thus, in addition to
the  monad $\mathcal{M}_{A,B}$, we also have the monads\begin{equation}\label{2nd quotient monad}
\mathcal{M}'_{\tau}:\ \ 0\to H_{n-r}\otimes\mathcal{O}_{\mathbb{P}^3}(-1)\overset{a'_\tau}\to
E_{2n}(B)\overset{\phi}{\underset{\simeq}\to}E_{2n}(B)^\vee
\overset{{a'}_{\tau}^\vee}\to H_{n-r}^\vee\otimes\mathcal{O}_{\mathbb{P}^3}(1)\to0,
\end{equation}
with cohomology
$$
E_{2r}(A_\tau)=\ker({a'}_{\tau}^\vee\circ \phi)/{\rm im\:} a'_\tau,
$$
and
\begin{equation}\label{3rd quotient monad}
\mathcal{M}''_{\tau}:\ \ 0\to H_{r-1}\otimes
\mathcal{O}_{\mathbb{P}^3}(-1)\overset{a''_\tau}\to
E_{2r}(A_\tau)\overset{\phi_\tau}{\underset{\simeq}\to}
E_{2r}(A_\tau)^\vee\overset{{a''}_{\tau}^\vee}\to H_{r-1}^\vee
\otimes\mathcal{O}_{\mathbb{P}^3}(1)\to0,
\end{equation}
with cohomology
$$
E_2(A)=\ker({a''}_{\tau}^\vee\circ \phi_\tau)/{\rm im\:}
a''_{\tau}.
$$
Since $E_{2n}(B)$ is a symplectic instanton, $h^0(E_{2n}(B))=h^i(E_{2n}(B)(-2))=0$, and the monad $\mathcal{M}'_{\tau}$ yields
$$h^0(E_{2r}(A_\tau))=h^i(E_{2r}(A_\tau)(-2))=0,\ i\ge0,\qquad c_2(E_{2r}(A_\tau))=2n-r+1.$$
This, together with \eqref{sympl E2r}, means that
\begin{equation}\label{belongs to E2r}
[E_{2r}(A_\tau)]\in I_{2n-r+1,r}.
\end{equation}

\begin{remark}\label{Dnr}
The maps $\tau$ lie in the  set
$$
N_{n,r}:=\{\tau\in{\rm Hom}(H_{2n-r+1},H_{2n})|\ \tau\
\textrm{is injective and}\ \mathrm{im}\ \tau
\supset\mathrm{im}\ i_\zeta\}
$$
which, for fixed  $A\in MI_{2n,1}(\zeta)$, parameterizes
a family of hyperwebs $A_\tau$ from $MI_{2n-r+1,r}$.
Now, $N_{n,r}$ is a principal $GL(H_{2n-r+1})$-bundle over an
open subset of the Grassmannian $Gr(n-r,n-1)$,  so   it is
irreducible. As a result, the family of the three-row extensions of the
diagram \eqref{twomonads} is parameterized by the
irreducible variety $MI_{2n,1}(\zeta)\times N_{n,r}$.
This in turn implies that the family $D_{n,r}$ of isomorphism classes of symplectic
rank-$2r$ bundles obtained from these diagrams by
\eqref{rk2n coho sheaf} is an irreducible, locally closed subset
of $I_{2n-r+1,r}$.
It is not clear  a priori if the closure of
$D_{n,r}$ in $I_{2n-r+1,r}$ is an irreducible component of
$I_{2n-r+1,r}$.
\end{remark}

Let $2\le r\le n-1$. For every monomorphism $i:H_n\hookrightarrow H_{2n-r+1}$,
denote by $B(A,i)$ the image of $A\in MI_{2n-r+1,r}$ under the projection
$\mathbf{S}_{2n-r+1}\twoheadrightarrow\mathbf{S}_n$
induced by $i$. It may be regarded as a homomorphism
$B(A,i):H_n\otimes V\to H_n^\vee\otimes V^\vee$.

\begin{definition}\label{property (*)}
We say that {\it $A\in MI_{2n-r+1,r}$
satisfies property} (*) if there exists a monomorphism
$i:H_n\hookrightarrow H_{2n-r+1}$
such that $B(A,i)$
is invertible.
\end{definition}
This  is   an open condition on $A$. By Theorem \ref{princbundle},
 $\pi_{2n-r+1,r}: MI_{2n-r+1,r}\to I_{2n-r+1,r}$
is a principal bundle, so that,
if an element $A\in\pi_{2n-r+1,r}^{-1}([E_{2r}])$ satisfies (*),
then any other point $A'\in\pi_{2n-r+1,r}^{-1}([E_{2r}])$
satisfies (*).
A symplectic instanton $E_{2r}$ from
$I_{2n-r+1,r}$  is said to be \textit{tame} if some
(hence all) $A\in\pi_{2n-r+1,r}^{-1}([E_{2r}])$ satisfies
property (*).  This is an open condition on
$[E_{2r}]\in I_{2n-r+1,r}$.

\begin{remark}\label{I*}
Using \eqref{contains}, we see that any
$[E_{2r}]\in D_{n,r}$ is tame. We define
$$
I_{2n-r+1,r}^*:=I_{(1)}\cup\ldots\cup I_{(k)},
$$
where $I_{(1)},\ldots, I_{(k)}$
are   the irreducible components of $I_{2n-r+1,r}$ whose
general points are tame symplectic instantons. As
$D_{n,r}\subset I_{2n-r+1,r}^*$ by definition,   $I_{2n-r+1,r}^*$
is nonempty. If we define
$MI_{2n-r+1,r}^*= \pi_{2n-r+1,r}^{-1}(I_{2n-r+1,r}^*)$,
then the map
$\pi_{2n-r+1,r}:MI_{2n-r+1,r}^*\to I_{2n-r+1,r}^*$
is a principal $GL(H_{2n-r+1})/\{\pm 1\}$-bundle.
\end{remark}

\bigskip

\section{Irreducibility of $I^*_{2n-r+1,r}$}\label{I^*nr}

\subsection{A dense open subset  of $MI^*_{2n-r+1,r}$} \label{Xnr and MI*}
We want to obtain the irreducibility of
$I^*_{n,r}$ by reducing it  to that of $X_{n,r}$, a dense open subset  of $MI^*_{2n-r+1,r}$.
The subset $X_{n,r}$ is a locally closed subset of the product
of an affine space and an affine cone over a Grassmannian.
Given an integer $n\ge1$, we define the dense open subset of $\mathbf{S}_n$
$$
\mathbf{S}^0_n:=\{A\in\mathbf{S}_n\ |\
A:H_n\otimes V\to H_n^\vee\otimes V^\vee
\ {\rm is\ an\ invertible\ map}\}.
$$

We need some more notation.
By definition, an element   $B\in\mathbf{S}^0_n$  is an invertible anti-self-dual map
$H_n\otimes V\to H_n^\vee\otimes V^\vee$.  Its inverse
$
B^{-1}:H_n^\vee\otimes V^\vee\to H_n\otimes V
$
is also anti-self-dual.
Consider the vector space
$\mathbf{\Sigma}_{n,r}:=H_{n-r+1}^\vee\otimes H_n^\vee\otimes
\wedge^2V^\vee$. An element $C\in\mathbf{\Sigma}_{n,r}$ can be
viewed as a linear map
$C:H_{n-r+1}\otimes V\to H_n^\vee\otimes V^\vee$,
and its dual
$C^\vee:H_n\otimes V\to H_{n-r+1}^\vee\otimes V^\vee$.
As the composition $C^\vee\circ B^{-1}\circ C$ is
anti-self-dual, we can consider it as an element of
$\wedge^2(H_{n-r+1}^\vee\otimes V^\vee)\simeq\mathbf{S}_{n-r+1}
\oplus\wedge^2H_{n-r+1}^\vee\otimes S^2V^\vee$
Thus the condition
$$
D-C^\vee\circ B^{-1}\circ C\in\mathbf{S}_{n-r+1},\ \ \ \
D\in\wedge^2(H_{n-r+1}^\vee\otimes V^\vee)
$$
makes sense.

Under an arbitrary direct sum decomposition
\begin{equation}\label{xi}
\xi:H_n\oplus H_{n-r+1}\overset{\sim}\to H_{2n-r+1}
\end{equation}
we can represent the hyperweb
$A\in\mathbf{S}_{2n-r+1}$, regarded as a homomorphism
$A:H_n\otimes V\oplus H_{n-r+1}\otimes V\to
H_n^\vee\otimes V^\vee\oplus H^\vee_{n-r+1}\otimes V^\vee$,
as the $(8n-4r+4)\times(8n-4r+4)$-matrix of homomorphisms
\begin{equation}\label{matrix A}
A=\left(\begin{array}{cc}
A_1(\xi) & A_2(\xi) \\
-A_2(\xi)^\vee & A_3(\xi)
\end{array}\right),
\end{equation}
where
$$
A_1(\xi)\in\mathbf{S}_n,\ \ \ A_2(\xi)\in\mathbf{\Sigma}_{n,r}:=
{\rm Hom}(H_n,{H}_{n-r+1}^\vee)\otimes\wedge^2V^\vee,\ \ \
A_3(\xi)\in\mathbf{S}_{n-r+1}.
$$
With this notation, the decomposition \eqref{xi} induces an isomorphism
\begin{equation}\label{tilde xi}
\tilde{\xi}:\ \mathbf{S}_{2n-r+1}\overset{\sim}\to\mathbf{S}_n
\oplus\mathbf{\Sigma}_{n,r}\oplus\mathbf{S}_{n-r+1},\ \ A\mapsto
(A_1(\xi),A_2(\xi),A_3(\xi)).
\end{equation}

Let $\mathrm{Isom}_{n,r}$ be the set of all isomorphisms $\xi$
in \eqref{xi}. According to Definition \ref{property (*)}, there
exists $\xi\in\mathrm{Isom}_{n,r}$ such that the set
$$
MI_{2n-r+1,r}^*(\xi):=\{A\in MI_{2n-r+1,r}\ |\ A\
\textrm{satisfies property}\ (*)\
\textrm{for the monomorphism}
$$
$$
i_{\xi}:H_n\hookrightarrow\
H_{2n-r+1}\ \textrm{determined by}\ \xi\}
$$
is a dense open subset of $MI^*_{2n-r+1,r}$. Now take
$A\in MI_{2n-r+1,r}^*(\xi)$ and consider $A$ as a matrix of
homomorphisms as in \eqref{matrix A}. By definition, the submatrix
$A_1(\xi)$ is invertible.
By a suitable elementary transformation we reduce
the matrix $A$ to an equivalent matrix $\tilde{A}$ of the form
$$
\tilde{A}=\left(\begin{array}{cc}
{\rm id}_{H_n\otimes V} & A_1(\xi)^{-1}\circ A_2(\xi) \\
0 & A_2(\xi)^\vee\circ A_1(\xi)^{-1}\circ A_2(\xi)+A_3(\xi)
\end{array}\right).
$$
Since $\mathrm{rk}\tilde{A}=\mathrm{rk} A=2(2n-r+1)+2r=4n+2$,
we obtain the following relation between the matrices
$A_1(\xi),\ A_2(\xi)$ and
$A_3(\xi)$:
\begin{equation}\label{rk(A3+...)}
\rk(A_2(\xi)^\vee\circ A_1(\xi)^{-1}\circ A_2(\xi)+A_3(\xi))=2.
\end{equation}

Consider the embedding of the Grassmannian
$$G:=Gr(2,H^\vee_{n-r+1}\otimes V^\vee)\hookrightarrow
P(\wedge^2(H^\vee_{n-r+1}\otimes V^\vee)),$$
and let $KG\subset\wedge^2(H^\vee_{n-r+1}\otimes V^\vee)$ be
the affine cone over $G$. Set $KG^*:=KG\smallsetminus\{0\}$.
We can now rewrite \eqref{rk(A3+...)} as
\begin{equation}\label{in KG*}
A_2(\xi)^\vee\circ A_1(\xi)^{-1}\circ A_2(\xi)+A_3(\xi)\in KG^*,
\end{equation}
where
\begin{equation}\label{A3 in Sn-r+1}
A_2(\xi)^\vee\circ A_1(\xi)^{-1}\circ A_2(\xi)\in
\wedge^2(H^\vee_{n-r+1}\otimes V^\vee),\ \ \ \
A_3(\xi)\in\mathbf{S}_{n-r+1}.
\end{equation}

Now consider the set
\begin{equation}\label{Ym}
\widetilde{X}_{n,r}:=\{(B,C,D)\in\mathbf{S}^0_n\times
\mathbf{\Sigma}_{n,r}\times KG^*\ |\ D-C^\vee\circ B^{-1}\circ
C\in\mathbf{S}_{n-r+1}\}.
\end{equation}
Since for an arbitrary point $y=(B,C,D)\in \tilde{X}_n$ the
point $\tilde{\xi}^{-1}(B,C,D-C^\vee\circ B^{-1}\circ C)$
lies in $\mathbf{S}_{2n-r+1}$,
it may be considered as a homomorphism
$A_y:\ H_{2n-r+1}\otimes V\to H_{2n-r+1}^\vee\otimes V^\vee$
of rank $4n+2$, and we have a well-defined $(4n+2)$-dimensional
vector space
$W_{4n+2}(y):=H_{2n-r+1}\otimes V/\ker A_y$
together with a canonical epimorphism
$c_y:\ H_{2n-r+1}\otimes V\twoheadrightarrow W_{4n+2}(y)$
and an induced skew-symmetric isomorphism
$q_y:\ W_{4n+2}(y)\overset{\sim}\to W_{4n+2}(y)^\vee$
such that
$A_y=c_y^\vee\circ q_y\circ c_y.$
Now, similarly to
the morphism $a_A:\ H_{2n-r+1}\otimes\mathcal{O}_{\mathbb{P}^3}(-1)\to W_{4n+2}\otimes\mathcal{O}_{\mathbb{P}^3}$
(see subsection \ref{n+1,n-instantons}), a morphism of sheaves
\begin{equation}\label{ay}
a_y=c_y\circ u:\ H_{2n-r+1}\otimes\mathcal{O}_{\mathbb{P}^3}(-1)
\to W_{4n+2}(y)\otimes\mathcal{O}_{\mathbb{P}^3}
\end{equation}
is defined, together with its transpose
${}^ta_y=a_y^\vee\circ q_y:\ W^\vee_{4n+2}(y)\otimes\mathcal{O}_{\mathbb{P}^3}
\to H_{2n-r+1}^\vee\otimes\mathcal{O}_{\mathbb{P}^3}(1)$.
We now introduce an open subset $X_{n,r}$ of the set
$\widetilde{X}_{n,r}$,
\begin{equation}\label{Y0m}
X_{n,r}:=\left\{y\in\widetilde{X}_{n,r}\ \left|\
\begin{matrix}
(i)\ {}^ta_y\ {\rm is\ epimorphic},\ \ \ \ \ \ \ \ \ \ \cr
(ii)\ [\ker{}^ta_y/{\rm im}a_y]\in I^*_{2n-r+1,r}
\end{matrix}
\right.
\right\}.
\end{equation}
Since the conditions (i) and (ii) on a point
$y\in\widetilde{X}_{n,r}$ in \eqref{Y0m} are open, from
\eqref{in KG*} and \eqref{A3 in Sn-r+1} we obtain the following
result.
\begin{proposition}\label{isomorphism fnr}
There exist a decomposition $\xi\in\mathrm{Isom}_{n,r}$, a
dense open subset\break $MI^*_{2n-r+1,r}(\xi)$ of $MI^*_{2n-r+1,r}$
and an isomorphism of reduced schemes
$$
f_{n,r}:MI^*_{2n-r+1,r}(\xi)\overset{\sim}\to X_{n,r},\
A\mapsto(A_1(\xi),A_2(\xi),A_3(\xi)).
$$
The inverse isomorphism is given by the formula
$$
f_{n,r}^{-1}:\ X_{n,r}\overset{\sim}\to MI^*_{2n-r+1,r}(\xi):
\ (B,C,D)\mapsto\ \tilde{\xi}^{-1}(B,\ C,\ D-C^\vee\circ
B^{-1}\circ C),
$$
where $\widetilde{\xi}$ is defined in \eqref{tilde xi}.
\end{proposition}

The following theorem will be proved in Subsection \ref{Irreducibility of Xnr}.

\begin{theorem}\label{Irred of Xnr}
$X_{n,r}$ is irreducible of dimension
$(2n-r+1)^2+4(2n-r+1)(r+1)-r(2r+1)$.
\end{theorem}

Proposition \ref{isomorphism fnr} and Theorem
\ref{Irred of Xnr}  imply that $MI_{2n-r+1,r}^*$ is
irreducible of dimension $(2n-r+1)^2+4(2n-r+1)(r+1)-r(2r+1)$
for any $n\le3$ and $2\le r\le n-1$. Thus,  for these
values of $n$ and $r$, the space
$I_{2n-r+1,r}^*$ is
irreducible and has dimension $4(2n-r+1)(r+1)-r(2r+1)$. Substituting $2n-r+1\mapsto n$, we obtain
the   main result of this paper.

\begin{theorem}\label{Irred of I*}
For any integer $r\ge2$ and for any integer $n\ge r-1$ such that
$n\equiv r-1({\rm mod}2)$, the moduli space $I_{n,r}^*$
of tame symplectic instantons is an open subset of an
irreducible component of $I_{n,r}$ of dimension
$4n(r+1)-r(2r+1)$.
\end{theorem}

\subsection{\bf Proof of the irreducibility of $X_{n,r}$.}\label{Irreducibility of Xnr}
We prove now Theorem
\ref{Irred of Xnr}.
Consider the set $\widetilde{X}_{n,r}$ defined in \eqref{Ym}.
Since $X_{n,r}$ is an open subset of $\widetilde{X}_{n,r}$, it
is enough to prove the irreducibility of $\widetilde{X}_{n,r}$.
In view of the isomorphism
$\mathbf{S}_n^0\overset{\sim}\to(\mathbf{S}_n^\vee)^0:\
B\mapsto B^{-1}$,
we rewrite $\widetilde{X}_{n,r}$ as
$$
\widetilde{X}_{n,r}=\{(B,C,D)\in(\mathbf{S}_n^\vee)^0\times
\mathbf{\Sigma}_{n,r}\times KG^*\ |\ D-C^\vee\circ B\circ
C\in\mathbf{S}_{n-r+1}\}.
$$
If a direct sum decomposition
$$
H_n\overset{\sim}\to H_{n-r+1}\oplus H_{r-1}
$$
has been fixed, any linear map
$$
C\in\mathbf{\Sigma}_{n,r}=\Hom(H_{n-r+1},H_n^\vee\otimes\wedge^2 V^\vee),\ \ \
C:H_{n-r+1}\otimes V\to H_n^\vee\otimes V^\vee,\ \ \ \
$$
can be represented as a homomorphism
$$
C:\ H_{n-r+1}\otimes V\to H_{n-r+1}^\vee\otimes V^\vee\ \oplus\ H_{r-1}^\vee\otimes V^\vee,
$$
and also as a block matrix
\begin{equation}\label{matrix of C}
C= \left(
\begin{array}{c}
\phi \\
\psi
\end{array}
\right),
\end{equation}
with
$$
\phi\in\Hom(H_{n-r+1},H_{n-r+1}^\vee)\otimes\wedge^2 V^\vee=
\mathbf{\Phi}_{n-r+1},\ \ \
\psi\in\mathbf{\Psi}_{n,r}:=
\Hom(H_{n-r+1},H_{r-1}^\vee)\otimes\wedge^2 V^\vee.
$$
In the same way, any
$D\in (\mathbf{S}^\vee_n)^0\subset\mathbf{S}^\vee_n= S^2H_n\
\otimes\wedge^2V\subset\Hom(H_n^\vee\otimes V^\vee,H_n\otimes V)$
can be represented  as
\begin{equation}\label{matrix of D}
B=\left(
\begin{array}{cc}
B_1 & \lambda \\
-\lambda^\vee & \mu
\end{array}
\right),
\end{equation}
with
\begin{equation}\label{A, lambda, mu}
B_1\in\mathbf{S}^\vee_{n-r+1}\subset
\Hom(H_{n-r+1}^\vee\otimes V^\vee,H_{n-r+1}\otimes V),\ \ \
\end{equation}
$$
\lambda\in\mathbf{L}_{n,r}:=\Hom(H_r^\vee,H_{n-r+1})\otimes
\wedge^2 V, \ \ \ \mu\in\mathbf{M}_{r-1}:=
S^2H_{r-1}\otimes\wedge^2V.
$$
By \eqref{matrix of C} and \eqref{matrix of D} the composition
$$
C^\vee\circ B\circ C:H_{n-r+1}\otimes V\to
H_{n-r+1}^\vee\otimes V^\vee\ \
(C^\vee\circ B\circ C\in\wedge^2(H_{n-r+1}^\vee\otimes V^\vee))
$$
can be written in the form
\begin{equation}\label{CDC}
C^\vee\circ B\circ C=\phi^\vee\circ B_1\circ\phi+\phi^\vee\circ
\lambda\circ\psi-\psi^\vee\circ\lambda^\vee\circ\phi+
\psi^\vee\circ\mu\circ\psi.
\end{equation}
In view of \eqref{matrix of C}-\eqref{A, lambda, mu} we have
$$
\mathbf{S}^\vee_n\times\mathbf{\Sigma}_{n,r}=
\mathbf{S}^\vee_{n-r+1}\times\mathbf{\Phi}_{n-r+1}
\times\mathbf{\Psi}_{n,r}\times\mathbf{L}_{n,r}\times
\mathbf{M}_{r-1},
$$
and   well-defined morphisms
$$
\tilde{p}:\widetilde{X}_{n,r}\to\mathbf{L}_{n,r}\times
\mathbf{M}_r\times KG,\ (B_1,\phi,\psi,\lambda,\mu,D)\mapsto
(\lambda,\mu,D).
$$
and
$$
p:=\tilde{p}|\overline{X}_{n,r}:\overline{X}_{n,r}\to
\mathbf{L}_{n,r}\times\mathbf{M}_{r-1}\times KG.
$$
Here $\overline{X}_{n,r}$ is the closure of
$\widetilde{X}_{n,r}$ in
$(\mathbf{S}_n^\vee)^0\times\mathbf{\Sigma}_{n,r}\times KG$.
Moreover, we have:

\begin{proposition}\label{nondeg for general}
Let $n\ge2$. For any $B\in(\mathbf{S}^\vee_n)^0$ and for
a general choice of the decomposition
$H_n\simeq H_{n-r+1}\oplus H_{r-1}$, the block
$B_1$ of $B$ in \eqref{matrix of D} is nondegenerate.
\end{proposition}

\begin{proof} By applying Proposition 7.3] in \cite{T1} $r$ times, one obtains a
decomposition
$H_n\overset{\sim}\to H_{n-r+1}\oplus H_{r-1}$
such that
$B_1:H_{n-r+1}^\vee\otimes V^\vee\to H_{n-r+1}\otimes V$ in
\eqref{matrix of D}
is nondegenerate, that is, $B_1\in(\mathbf{S}^\vee_{n-r+1})^0$.
\end{proof}

If $\mathcal{X}$ is any irreducible component of $X_{n,r}$,
taken with its   reduced structure, and
$\overline{\mathcal{X}}$  is its closure in
$\overline{X}_{n,r}$, we pick up  a point
$z=(B_1,\phi,\psi,\lambda,\mu,D)\in \mathcal{X}$
not lying in the components of $X_{n,r}$ different from
$\mathcal{X}$, and such that the decomposition
$H_n\simeq H_{n-r+1}\oplus H_{r-1}$ is general.
Then, by Proposition \ref{nondeg for general}, $B_1\in(\mathbf{S}^\vee_{n-r+1})^0$.
Consider the morphism
$$
f:\ \mathbb{A}^1\to\overline{\mathcal{X}},\
t\mapsto(B_1,t^2\phi,t\psi,t\lambda,t^2\mu,t^4D),\ \ \ f(1)=z.
$$
This  is well defined as a consequence of  \eqref{CDC}. The point
$f(0)=(B_1,0,0,0,0,0)$ lies in the fibre $p^{-1}(0,0,0)$, so that
$p^{-1}(0,0,0)\cap\overline{\mathcal{X}}\ne\emptyset$.
In different terms,
\begin{equation}\label{nonempty fibre}
\rho^{-1}(0,0,0)\ne\emptyset,\ \ \ \ \textrm{where}\ \ \
\rho:=p|\overline{\mathcal{X}}.
\end{equation}
By \eqref{CDC} and the definition of
$\widetilde{X}_{n,r}$, one has
\begin{equation}\label{zero fibre}
\tilde{p}^{-1}(0,0,0)=\{(B_1,\phi,\psi)\in
(\mathbf{S}^\vee_{n-r+1})^0\times\mathbf{\Phi}_{n-r+1}\times
\mathbf{\Psi}_{n,r}\ |\
\phi^\vee\circ B_1\circ\phi\in\mathbf{S}_{n-r+1}\}.
\end{equation}

Now for each $i\ge1$ consider the set $Z_i$ mentioned in the introduction. This set $Z_i$ is defined in
\cite[Section 7]{T1} as
\begin{equation}\label{zi}
Z_i=\{(B,\phi)\in(\mathbf{S}^\vee_i)^0\times\mathbf{\Phi}_i\ |\
\phi^\vee\circ B\circ\phi\in\mathbf{S}_i\},
\end{equation}
and   has a natural structure of   closed subscheme of
$(\mathbf{S}^\vee_i)^0\times\mathbf{\Phi}_i$
The key point in the sequel is the fact that $Z_i$ is an integral scheme of dimension 
$4i(i+2)$---see \cite[Theorem 7.2]{T1}. This statement is based on the following relation between $Z_i$ for $i\ge2$ and the moduli space of 't Hooft instantons of charge $2i-1$. Fix a monomorphism
$j:H_{i-1}\hookrightarrow H_i$. For an arbitrary point $z=(B,\phi)\in Z_i$, let $E_{2i}$ be a symplectic
vector bundle of rank $2i$ defined as a cokernel of a morphism of sheaves
$\tilde{B}:H_i\otimes\mathcal{O}_{\mathbb{P}^3}(-1)\to H_i^{\vee}\otimes\Omega_{\mathbb{P}^3}(1)$
naturally induced by $B$. Let $s(z):H_i\to H^0(E_{2i}(1))$ be the composition of $\phi$ understood as
a homomorphism $H_i\to H_i^{\vee}\otimes\wedge^2V^{\vee}$
and of the evaluation map $H_i^{\vee}\otimes\wedge^2V^{\vee}\to H^0(E_{2i}(1))$,
and let $s_z$ be the composition
$s_z:H_i\otimes\mathcal{O}_{\mathbb{P}^3}(-1)\overset{s(z)}{\to}
H^0(E_{2i}(1))\otimes\mathcal{O}_{\mathbb{P}^3}(-1)\overset{ev}{\to}E_{2i}$,
where $ev$ is the evaluation morphism. Using the symplecticity of $E_{2i}$, one obtains an antiselfdual monad
$M(z):0\to H_{i-1}\otimes\mathcal{O}_{\mathbb{P}^3}(-1)\overset{s_z\circ j}{\to}E_{2i}\overset{{}^t(s_z\circ j)}{\to}
H_{i-1}^{\vee}\otimes\mathcal{O}_{\mathbb{P}^3}(1)\to0$ with a rank-2 cohomology vector bundle $E_2(z)$ with
$c_1=0$ and $c_2=2i-1$. A standard diagram chase yields a monomorphism
$H_i/j(H_{i-1})\otimes\mathcal{O}_{\mathbb{P}^3}(-1)\to E_2(z)$
showing that $h^0(E_2(z)(1))\ne0$, i.~e. that $E_2(z)$ is a 't Hooft instanton vector bundle. Thus the association
$z\rightsquigarrow M(z)$ yields a morphism of $Z_i$ to the space $M^{tH}_{2i-1}$ of the 't Hooft monads, which is
irreducible since the moduli space of 't Hooft instantons of charge $2i-1$ is known to be irreducible. It is
shown in \cite[Section 9]{T1} that this morphism $Z_i\to M^{tH}_{2i-1}$ is a composition of a dense open embedding and
the structure map of an affine bundle over  $M^{tH}_{2i-1}$. This implies the irreducibility of $Z_i$.

Now, comparing (\ref{zi}) for $i=n-r+1$ with (\ref{zero fibre}), 
we obtain scheme-theoretic inclusions
\begin{equation}\label{zero fibre subset Zm times Psim}
\rho^{-1}(0,0,0)\subset p^{-1}(0,0,0)\subset
\tilde{p}^{-1}(0,0,0)=Z_{n-r+1}\times\mathbf{\Psi}_{n,r}.
\end{equation} 
By the above, $Z_{n-r+1}$ is an integral scheme of
dimension $4(n-r+1)(n-r+3)$. This together with
\eqref{zero fibre subset Zm times Psim} implies that
\begin{equation}\label{dim fibre le dim Zm+dim Psim}
\dim\rho^{-1}(0,0,0)\le\dim p^{-1}(0,0,0)\le\dim Z_{n-r+1}+
\dim\mathbf{\Psi}_{n,r}=4(n-r+1)(n-r+3)
\end{equation}
$$
+6(r-1)(n-r+1)=(n-r+1)(4n+2r+6).
$$
Hence, in view of \eqref{nonempty fibre},
\begin{multline}\label{dim X le}
\dim \overline{\mathcal{X}}\le\dim\rho^{-1}(0,0,0)+
\dim\mathbf{L}_{n,r}+\dim\mathbf{M}_{r-1}+\dim KG \\ \le
(n-r+1)(4n+2r+6)
+6(r-1)(n-r+1)+3(r-1)r+(8n-8r+5) \\ =(2n-r+1)^2+4(2n-r+1)(r+1)
-r(2r+1).
\end{multline}
On the other hand, formula \eqref{dim MInr ge...}---with $n$ replaced by $2n-r+1$---and
Proposition \ref{isomorphism fnr} show that, for any point $x\in\mathcal{X}$ such that
$A:=f_{n,r}^{-1}(x)\in MI_{2n-r+1,r}^*(\xi)$,
\begin{equation}\label{dim X ge}
(2n-r+1)^2+4(2n-r+1)(r+1)-r(2r+1)\le\dim_A MI_{2n-r+1,r}^*(\xi)=
\dim\overline{\mathcal{X}}.
\end{equation}
Comparing \eqref{dim X le} with \eqref{dim X ge}, we see that
all the inequalities in
\eqref{dim fibre le dim Zm+dim Psim}--\eqref{dim X ge}
are equalities. In particular,
\begin{equation}\label{dim fibre=dim X'-dim base}
\dim\rho^{-1}(0,0)=\dim(Z_{n-r+1}\times\mathbf{\Psi}_{n,r})=
\dim\overline{\mathcal{X}}-\dim(\mathbf{L}_{n,r}
\times\mathbf{M}_{r-1}\times KG).
\end{equation}
Since, by Theorem \cite[Theorem 7.2]{T1}, the scheme $Z_{n-r+1}$
is integral and so $Z_{n-r+1}\times\mathbf{\Psi}_{n,r}$ is
integral as well, \eqref{zero fibre subset Zm times Psim} and
\eqref{dim fibre=dim X'-dim base} yield the coincidence
of the integral schemes
\begin{equation}\label{zero fibre = Zm times Psim}
\rho^{-1}(0,0,0)=p^{-1}(0,0,0)=
\tilde{p}^{-1}(0,0,0)=Z_{n-r+1}\times\mathbf{\Psi}_{n,r}.
\end{equation}

We need now the following easy Lemma, which is a slight
generalization of Lemma 7.4 from \cite{T1}.  \begin{lemma}\label{flat implies irred}
Let $f:X\to Y$ be a morphism of reduced schemes, with $Y$   an
integral scheme. Assume that there exists a closed point
$y\in Y$ such that, for any irreducible component $X'$ of $X$,

(a) $\dim f^{-1}(y)=\dim X'-\dim Y$,

(b) the scheme-theoretic inclusion of fibres
$(f|_{X'})^{-1}(y)\subset f^{-1}(y)$
is an isomorphism of integral schemes.\\
Then

(i) there exists an open subset $U$ of $Y$ containing $y$ such
that the morphism $f|_{f^{-1}(U)}:f^{-1}(U)\to U$ is flat, and

(ii) $X$ is integral.
\end{lemma}

By applying   this lemma to
$X=X_{n,r},\ X'=\mathcal{X},\ Y=\mathbf{L}_{n,r}\times
\mathbf{M}_{r-1}\times KG,\ y=(0,0),f=p$, also in view of
\eqref{dim fibre=dim X'-dim base} and
\eqref{zero fibre = Zm times Psim},
one obtains that $X_{n,r}$ is integral and is of dimension
$$(2n-r+1)^2+4(2n-r+1)(r+1)-r(2r+1).$$
Theorem \ref{Irred of Xnr} is thus proved.

\end{document}